\documentclass[12pt]{article} 
\usepackage[utf8]{inputenc}
\usepackage[T2A]{fontenc}\usepackage[english]{babel}
\usepackage{amsthm, amsmath, amssymb, amsbsy, amscd, amsfonts, latexsym, euscript,epsf,stackrel}
\usepackage{authblk}
\usepackage{graphicx} 
\usepackage{epsfig}

\usepackage{tikz}
\usetikzlibrary{positioning,arrows}

\newtheorem{theorem}{Theorem}[section]

\newtheorem{proposition}[theorem]{Proposition}
\theoremstyle{definition}

\newtheorem{definition}[theorem]{Definition}

\newtheorem{remark}[theorem]{Remark}
\numberwithin{equation}{subsection}
\newtheorem*{ack}{Acknowledgement}
\usepackage[all,cmtip]{xy}
\usepackage{xcolor}
\usepackage{xfrac}
\usepackage{pb-diagram}
\usepackage{graphicx}
\usepackage{float}

\newcommand{\T}{\operatorname{T}}

\newcommand{\id}{\mathrm{id}}

\makeatletter
\newcommand{\subjclass}[2][1991]{%
	\let\@oldtitle\@title%
	\gdef\@title{\@oldtitle\footnotetext{#1 \emph{Mathematics subject classification.} #2}}%
}
\newcommand{\keywords}[1]{%
	\let\@@oldtitle\@title%
	\gdef\@title{\@@oldtitle\footnotetext{\emph{Key words.} #1.}}%
}
\makeatother

\def\bea{\begin{eqnarray}}
	\def\eea{\end{eqnarray}}
\def\nn{\nonumber}

\begin{document}
	\title{2-dimensional self-distributive non-counital bialgebras and knot invariants}

	\author[1,2,3]{Valeriy G. Bardakov\footnote{bardakov@math.nsc.ru}}
	\author[3]{Tatiana A. Kozlovskaya\footnote{konus\_magadan@mail.ru}}
	\author[1]{Alexander~S.~Panasenko\footnote{a.panasenko@g.nsu.ru}}
	\author[4,5]{Dmitry~V.~Talalaev\footnote{dtalalaev@yandex.ru}}
	\affil[1]{\small Sobolev Institute of Mathematics, Novosibirsk 630090, Russia.
	}
	\affil[2]{\small Novosibirsk State Agrarian University, Dobrolyubova street, 160, Novosibirsk, 630039, Russia.
	}
	\affil[3]{\small Regional Scientific and Educational Mathematical Center of Tomsk State University,
		36 Lenin Ave., 14, 634050, Tomsk, Russia.  
}
\affil[4]{\small Lomonosov Moscow State University, 
119991, Moscow, Russia.
}
\affil[5]{\small Center of Integrable Systems, Demidov Yaroslavl State University, Yaroslavl, Russia, 150003, Sovetskaya Str. 14
}

	\date{\today}

	%\subjclass[2010]{Primary 17D99; Secondary 57M27, 16S34, 20N02}
	\keywords{Algebra, coalgebra, coassociative coalgebra, bialgebra, self-distributive bialgebra. }

	\maketitle
	
	\begin{abstract}	 
In 	the preprint of  V. Bardakov, T. Kozlovskaya, D. Talalaev (Self-distributive bialgebras, arXiv:2501.19152) it was formulated a problem of classification of self-dist\-ribu\-tive bialgebras and was given classification of two-dimensional  counital self-distributive bialgebras.
		In this paper, we consider non-counital case. We find all 2-dimensional algebras of this type.
	 In constructed algebras we study the question of finding quandles for constructing knot invariants.   This activity is part of the overall program for the linearization of the concepts of rack and quandle and the development of the representations theory of  these structures.
	\end{abstract}

	\tableofcontents

	\section{Introduction}
	
	By groupoid we mean a non-empty set with one binary algebraic operation.
 The classical examples are semigroup and group, which are associative groupoids. An alternative example was invented by D. Joyce  \cite{J} and S. Matveev \cite{M} in the context of the theory of knot invariants, namely the structures of racks and quandles. These groupoids are self-distributive. Recall that a groupoid  $(X, *)$ is said to be a right self-distributive groupoid if for any $u, v, w \in X$ the following equality holds
$$
(u*v)*w = (u*w)*(v*w).
$$ 
By symmetry, a groupoid  $(X, *)$ is said to be a left self-distributive groupoid if for any $u, v, w \in X$ one has
$$
w*(u*v) = (w*u)*(w*v).
$$ 
In the present paper we are studying right self-distributive groupoid, we call them simply self-distributive groupoid.

The notion of the group algebra is very important in the theory of  groups, in particular in its representation theory. If $(G, \cdot)$ is a group, $\Bbbk$ is a commutative associative ring with unit, then the group algebra $\Bbbk[G]$ is the set of finite formal linear combinations of elements of $G$ with a natural distributive multiplication. The associativity follows from the fact that the  associativity identity in $G$ 
$$
(ab)c = a(bc)
$$
is extended to $\Bbbk[G]$ as a tri-linear operation. It is quite natural to generalize this idea to racks and quandles. A theory of quandle algebras analogous to the classical theory of group algebras has been proposed in \cite{BPS}, where several relations between quandles and their associated quandle	algebras have been explored. 
 
It is easy to see that the axiom of self-distributivity  is non-linear in $w$ and quandle algebras of non-trivial quandles are non-self-distributive.  
By analogy with self-distributive groupoid, an algebra $A$ over a field $\Bbbk$ is said to be self-distributive algebra if
	$$
	(a b) c = (a c) (b c),~~~a, b, c \in A.
	$$
This structure appears to be quite poor. Let $A$ be a self-distributive algebra. If $\alpha\in\Bbbk$ then $(ab)(\alpha c) = (a\cdot \alpha c)(b\cdot \alpha c) = \alpha^2 (ab)c$. It means that $(\alpha^2-\alpha)((ab)c) = 0$. In particular this implies that if $\Bbbk\neq \mathbb{Z}_2$ then a self-distributive algebra over a field $\Bbbk$ is right-nilpotent of index $3$, i.e. $(ab)c=0$ for any $a,b,c\in A$. A complete description of self-distributive algebras can be found in \cite{BKT1}.

Zero-divisors in quandle rings have been investigated in \cite{BPS} using the idea of orderability
of quandles. It has been proved that quandle rings of left or right orderable quandles
which are semi-latin have no zero-divisors. As a special case, it follows that quandle
rings of free quandles have no zero-divisors. It is well known that units play a fundamental role in the structure theory of group rings. In contrast, idempotents are the
most natural objects in quandle rings since each quandle element is, by definition, an
idempotent of the quandle ring.

It is well known that the group algebra caries a structure of Hopf algebra. The principal idea of \cite{BKT1} is to study a bialgebra structure of rack algebras and their generalization with non group-like comultiplication. We hope in future to develop this idea in the flavour of quantum groups. 

Typically a bialgebra is an associative algebra with consistent comultiplication, but in last few decades it appeared in literature non-associative bialgebras. The principal example for us is bialgebra defined on a rack algebra. 
A notion of a self-distributive bialgebra was introduced and elaborated in \cite{ABRW} in the case of rack algebras. It was demonstrated that just extending the group-like comultiplication on rack elements to their linear combinations by linearity one gets a self-distributive bialgebra.

The purpose of this computational paper is two fold. The first one  is to give a classification of 2-dimension non-counital self-distributive bialgebra. 
The second one is description of idempotents and quandles in these 2-dimension algebras. 

As we know quandles give strong invariants of links in classical knot theory. In \cite{ES} quandle algebras are used for construction of knot invariants and 
was shown that quandle rings and their idempotents lead to proper enhancements of the well-known quandle coloring invariant of links in the 3-space.
Also, in this paper the authors considered quandles of small orders, constructed quandle algebras and find quandles in this algebras. In more cases these quandles are trivial.
In \cite{ES} was  presented computer assisted calculations
of all idempotents for integral as well as mod 2 quandle rings of all quandles of order
less than six, and also determine quandles for which the set of all idempotents is itself
a quandle.

The present article is organize by the following manner.  In Section \ref{prel} we recall some definitions of coalgebras, bialgebras,  quandles and quandle algebras, which can be found in \cite{Kas, J, M, BPS, BPS-1}. 
Section \ref{class} is devoted to the classification of 2-dimensional non-counital self-distributive bialgebras. 
In section \ref{id} we find idempotents and quandles in algebar which was constructed in Section \ref{prel}.
Also, we show that for some algebras the augmentation map $\varepsilon \colon A \to \Bbbk$ is not a homomorphism and there is an idempotent $u \in A$ for which $\varepsilon (u)$ is arbitrary element in $\Bbbk$. Recall that if $\Bbbk$ is an integral domain with unity, then the augmentation map $\varepsilon \colon \Bbbk[Q] \to \Bbbk$ is a ring homomorphism and each idempotent of $\Bbbk[Q]$ has augmentation value 0 or 1.

%Let $L$ be a link in $\mathbb{R}^3$, $Q(L)$ its link quandle and $X$ a quandle. Then the number of quandle homomorphisms $|{\rm Hom}(Q(L), X)|$ is an invariant of $L$ called the quandle coloring invariant. This invariant generalises the well-known Fox coloring invariant of
%links.	
%Let X and Y be quandles and ${\rm Hom}_{alg} (\Bbbk[X], \Bbbk[Y])$ denotes the set of $\Bbbk$--algebra homomorphisms from $\Bbbk[X]$ to  $\Bbbk[Y]$.
%In particular, the author found the set of idempotent $Id([X])$ and if they are a quandle, then they take this quandle as $Y$.
%In \cite{ES} was proved that the set $|{\rm Hom}_{alg} (\Bbbk[X], \Bbbk[Y])|$ is a link invariant which is stronger than invariant $|{\rm Hom}(Q(L), X)|$. 

	\bigskip
	
	%%%%%%%%%%%%%%%%%%%%%%%%%%%%%%%%%%%%%%%%%%%%%%%%%%%%%%%%%%%%%%%%%%%%
	
	\section{Preliminaries} \label{prel}

	\subsection{Coalgebra}
	Let $C$ be a module over a commutative associative unital ring $\Bbbk$.   Comultiplication on
	$C$ is a linear map 
	$$
	\Delta \colon C \to C \otimes_\Bbbk C = C \otimes C.
	$$
For a finite dimensional $C$  a comultiplication on $C$ induces a multiplication on $C^*$. Hence $C^*$ becomes an algebra (not-associative in general).

	Using the  Sweedler notations we shall write  the comultiplication on $c \in V$ in the form
	\bea
	\Delta(c)=\sum_i c_i^{(1)} \otimes c_i^{(2)} =c^{(1)}\otimes c^{(2)}.\nn
	\eea
Comultiplication $\Delta$ is called a coassociative if the following holds in $C \otimes C \otimes C$:
$$
(\Delta \otimes \id_C ) \circ  \Delta  = (\id_C \otimes \Delta ) \circ  \Delta. 
$$ 
The pair $(C, \Delta)$ is called a (coassociative) coalgebra over $\Bbbk$. The coalgebra $(C, \Delta)$ is called cocommutative  if $\tau \circ  \Delta =  \Delta$, where
$\tau \colon C \otimes C \to  C \otimes C$, $\tau(a \otimes b) = b  \otimes a$ is the canonical flip map. A linear map $\varepsilon \colon C \to \Bbbk$  is called a counit for the coalgebra $(C, \Delta)$ if 
	$$
	(\varepsilon \otimes \id_C) \circ  \Delta  = (\id_C \otimes \varepsilon) \circ  \Delta = \id_C,~~~\Bbbk \otimes C \cong C. 
	$$ 
	The triple $(C, \Delta, \varepsilon)$ is called a counital coalgebra.

In this article we shall use the following definition, which was introduced in \cite{BKT1}.

	\begin{definition} 
A triple $(A, \Delta, \mu)$	is said to be a {\it bialgebra} if $(A, \mu)$ is an  $\Bbbk$-algebra (may be non-associative and without unit), 
$(A, \Delta)$ is a coassociative coalgebra and  a coassociative comultiplication
		$$
		\Delta \colon A \to A \otimes A
		$$
is a homomorphisms of $\Bbbk$-algebras, where the product on the right hand side is the component-vice product on $A\otimes A.$	
A bialgebra $(A, \Delta, \mu)$	is said to be a {\it self-distributive} if the following identities hold for all $a, b, c \in A$ 
				$$
		(a * b) * c = \sum_{i} (a * c^{(1)}_i) * (b * c^{(2)}_i)= (a * c^{(1)}) * (b * c^{(2)}),
		$$
where we write for all $a, b \in  A$ $b * a = \mu(a \otimes b))$	.
	\end{definition}
We call the condition that $\Delta$ is a homomorphisms of $\Bbbk$-algebras a 
consistency condition.

Self-distributive bialgebras were introduced in \cite{E} (see  also \cite{ABRW}).

	\subsection{Rack and quandles}
	
	It is well--known that quandle axioms are simply algebraic formulations of the three Reidemeister moves of planar diagrams of knots and links in
the 3-space \cite{J, M}, which therefore makes them suitable for defining invariants of
knots and links.

	A {\it quandle} is a non-empty set $Q$ with a binary operation $(x,y) \mapsto x * y$ satisfying the following axioms:
	\begin{enumerate}
		\item[(Q1)] For all $x \in Q$ holds $x*x=x$, i. e.  any element is idempotent,
		\item[(Q2)] For any $x,y \in Q$ there exists a unique $z \in Q$ such that $x=z*y$,
		\item[(Q3)] $(x*y)*z=(x*z) * (y*z)$ for all $x,y,z \in Q$.
	\end{enumerate}

	An algebraic system satisfying only (Q2) and (Q3)  is called a {\it rack}. 
	
	A quandle  $Q$ is called {\it trivial} if $x*y=x$ for all $x, y \in Q$.  Unlike groups, a trivial quandle can have arbitrary number of elements. We denote the $n$-element trivial quandle by $\T_n$ and an arbitrary trivial quandle by $\T$.
	\medskip

	\subsection{Rack algebras and quandle algebras}
	
	Rack algebra and quandle algebra notions were introduced in \cite{BPS}. 	
	Let $Q$ be a quandle and  $\Bbbk[Q]$ be the set of all formal finite $\Bbbk$-linear combinations of elements of $Q$, that is,
	$$\Bbbk[Q]:=\Big\{ \sum_i\alpha_i x_i~|~\alpha_i \in \Bbbk,~ x_i \in Q \Big\}.$$
	Then $\Bbbk[Q]$ is an additive abelian group with coefficient-wise addition. Define multiplication in $\Bbbk[Q]$ by setting $$\big(\sum_i\alpha_i x_i\big) \big(\sum_j\beta_j x_j\big):=\sum_{i,j}\alpha_i\beta_j (x_i x_j).$$
	Clearly, the multiplication is distributive with respect to addition from both left and right, and $\Bbbk[Q]$ forms  an $\Bbbk$-algebra, which we call the quandle algebra of $Q$ with coefficients in the ring $\Bbbk$. Since $Q$ is non-associative, unless it is a trivial quandle, it follows that $\Bbbk[Q]$ is a non-associative  in general. If $Q$ is a rack, then its  rack algebra $\Bbbk[Q]$ is defined analogously.  
	\par

	A non-zero element $u \in \Bbbk[Q]$ is called an idempotent if $u^2 = u$.  Clearly, elements of the basis $x \in Q $ are idempotents of
$\Bbbk[Q]$ and referred as trivial idempotents.
	
	Define the augmentation map
	$$\varepsilon \colon \Bbbk[Q] \to \Bbbk$$
	by setting $$\varepsilon \big(\sum_i\alpha_i x_i\big)= \sum_i\alpha_i .$$
	Clearly, $\varepsilon$ is a surjective  ring homomorphism, and $I_\Bbbk(Q):= \ker(\varepsilon)$ is a two-sided ideal of $\Bbbk[Q]$, called the {\it augmentation ideal} of $\Bbbk[Q]$. It is easy to see that $\{x-y~|~x, y \in Q \}$ is  a generating set for $I_\Bbbk(Q)$ as an $\Bbbk$-module. Further, if $x_0 \in Q$ is a fixed element, then the set $\big\{x-x_0~|~x \in Q \setminus \{ x_0\} \big\}$ is a basis for $I_\Bbbk(Q)$ as a $\Bbbk$-module. 
	\par

In \cite{BPS-1} has been shown that quandle rings of non--trivial quandles over rings of characteristic
other than two and three cannot be alternative or Jordan algebras.

\medskip

%In particular, a rack algebra $A = \Bbbk[X]$ is a self-distributive bialgebra if we define a coproduct $\Delta \colon X \to X \otimes X$ by the rule $\Delta(x) =  x \otimes x$ and extend it by linearity to $\Bbbk[X]$ (see \cite{CCES} and \cite{ABRW}).

The main examples of self-distributive bialgebras can be constructed on rack algebras (see \cite{E} and \cite{ABRW}). Let us define a comultiplication $\Delta \colon \Bbbk[X] \to \Bbbk[X] \otimes \Bbbk[X]$ by the rule $\Delta(x) = x \otimes x$ for any $x \in X$ and extend it to $\Bbbk[X]$ by linearity. It is easy to see that this comultiplication is coassociative. Further, let us define a map $\varepsilon \colon \Bbbk[X] \to \Bbbk$ by $\varepsilon(x) = 1$ for $x \in X$ and extend by linearity to $\Bbbk[X]$. We arrive to

\begin{proposition}
$(\Bbbk[X], \cdot, \Delta,  \varepsilon)$ is a counital self-distributive bialgebra.
\end{proposition}

	\bigskip

	%%%%%%%%%%%%%%%%%%%%%%%%%%%%%%%%%%%%%%%%%%%%%%%%%%%%%%%%%%%%%%%%%%
	
	\section{2-dimensional self-distributive non-counital bialgebras}
	\label{class}
	
	\medskip
	In \cite{BKT1} we formulated a question on constructing of self-distributive bialgebras and  gave full classification of such bialgebras with counit in dimension 2 over $\mathbb{C}$ and constructed some such algebras over arbitrary field. Similar results can be found in \cite{E}.

In this section we give a full description of 2-dimensional non-counital  self-distributive bialgebras over arbitrary field.

	\subsection{From  2-dimensional associative algebras to coassotiative coalgebras}
	
		To construct coassotiative coalgebras we take 2-dimensional non-unital associative algebras and go to the dual algebra. 
	
	\begin{proposition}  
	\label{2-dim-alg}
		Let $A$ be a 2-dimensional associative algebra without unit over a field $\Bbbk$. Then there exists a basis $\{e_1, e_2 \}$ in $A$ such that one of the following cases holds:
		
		1) $e_1^2 = e_1,~~e_1 e_2 = 0,~~e_2 e_1 = e_2,~~e_2^2 =0$;
		
		2) $e_1^2 = e_1,~~e_1 e_2 = e_2,~~e_2 e_1 = 0,~~e_2^2 =0$;
		
		3) $e_1^2 = e_1,~~e_1 e_2 = 0,~~e_2 e_1 = 0,~~e_2^2 =0$;
		
		4) $e_1^2 = e_2,~~e_1 e_2 = 0,~~e_2 e_1 = 0,~~e_2^2 =0$;
		
		5) $e_1^2 = 0,~~e_1 e_2 = 0,~~e_2 e_1 = 0,~~e_2^2 =0$.
		
	\end{proposition}
	
	\begin{proof} Let $A$ be a two-dimensional associative algebra over a field $\Bbbk$ that does not contain a unit. Since every semisimple finite-dimensional algebra contains a unit, the radical $R=R(A)$ is nonzero. We have $R^3 = 0$ because $\dim R \le 2$. Let us consider two cases.
	
	\medskip 1) $\mathrm{dim}(R) = 1$. Then $R^2 = 0$ and $R = \Bbbk e_2$. The algebra $A/R$ is a one-dimensional semisimple associative algebra, i.e. it is isomorphic to the field $\Bbbk$. Thus, $A$ contains an element $e_1\notin R$ such that $e_1^2 = e_1 + \alpha e_2$ for some $\alpha\in \Bbbk$. In addition, $e_1e_2 = \beta e_2$, $e_2e_1 = \gamma e_2$. Then
	\[
	\beta^2 e_2 = \beta e_1e_2 = e_1(e_1e_2) = e_1^2e_2 = (e_1+\alpha e_2)e_2 = \beta e_2,
	\]
	whence $\beta^2 = \beta$. Similarly, from the equality $(e_2e_1)e_1 = e_2e_1^2$ we obtain $\gamma^2 = \gamma$. Let us consider several subcases.
	
	\medskip 1.a) $\beta = \gamma = 1$. Then $\alpha\neq 0$ (otherwise $e_1$ is the unit of the algebra $A$). We can assume that $\alpha = 1$ (considering $\alpha e_2$ instead of $e_2$). But then the element $f = e_2-e_1$ is the identity of the algebra $A$, we have a contradiction.
	
	1.b) $\beta = 0$, $\gamma = 1$. We have $e_1^2e_1 = (e_1+\alpha e_2)e_1 = e_1+2\alpha e_2$, $e_1e_1^2 = e_1(e_1+\alpha e_2) = e_1+\alpha e_2$, whence $\alpha = 0$. We obtain the following multiplication table:
	\[e_1^2 = e_1, \quad e_1e_2 = 0, \quad e_2e_1 = e_2, \quad e_2^2 = 0.\]
	
	1.c) $\beta = 1$, $\gamma = 0$. Similar to the previous case, we obtain the following multiplication table:
	\[e_1^2 = e_1, \quad e_1e_2 = e_2, \quad e_2e_1 = 0, \quad e_2^2 = 0.\]
	
	1.d) $\beta = \gamma = 0$. Then the element $f_1 = e_1+\alpha e_2$ gives the following multiplication table:
	\[f_1^2 = f_1, \quad f_1e_2 = e_2f_1 = e_2^2 = 0.\]
	
	\medskip 2) $\mathrm{dim}(R) = 2$. Then $A = R$ and $A^3 = 0$. Suppose that $A^2\neq 0$. Then $A^2 = \Bbbk e_2$ and $A=\Bbbk e_1+ \Bbbk e_2$ for some $e_1,e_2\in A$.  Note that $e_2^2 = e_2e_1 = e_1e_2 = 0$, since $e_2^2, e_1e_2,e_2e_1\in A^3$. Moreover, $e_1^2 = \alpha e_2$. If $\alpha = 0$, then $A^2 = 0$. If $\alpha \neq 0$, then we can assume that $\alpha = 1$ and we get the following multiplication table:
	\[e_1^2 = e_2, \quad e_1e_2 = e_2e_1 = e_2^2 = 0.\]
	\end{proof}
	
\begin{remark}
The first algebra in this proposition is an algebra of the trivial quandle $T_2$. Indeed, if $T_2 = \{ t_1, t_2 \}$, then the elements $t = t_1$ and $\tau = t_1 - t_2$ form a basis of $\Bbbk[T_2]$. The multiplication table for this basis  has the form
$$
t^2 = t,~~\tau^2 = 0,~~ t \tau = 0, ~~\tau t = \tau.
$$
We have the algebra 1) from 
 Proposition \ref{2-dim-alg}.

\end{remark}
	
\medskip
	
	Using these algebras, we construct 2-dimensional coassotiative coalgebras
	$A^* = \langle f_1, f_2 \rangle$ which are  dual algebras to the algebras with basis $e_1, e_2$. The last algebra  is associative and since we assume that it does not have the unit, it has a multiplication of 1) -- 5). 
	Then, by the definition of dual basis, we have
	$$
	f_1(e_1) = 1, f_1(e_2) = 0,~~f_2(e_1) = 0, f_2(e_2) = 1.
	$$
	By the definition of comultiplication,
	$$
	\Delta f_k (e_i, e_j) = f_k (e_i e_j),~~~i, j, k \in \{ 1, 2 \},
	$$
	Further, we are considering all possible cases.
	
	1) Multiplication 
	$$
	e_1^2 = e_1,~~e_1 e_2 = 0,~~e_2 e_1 = e_2,~~e_2^2 =0.
	$$
	In this case we have the formulas for $f_1$,
	$$
	\Delta f_1 (e_1,  e_1) = f_1(e_1) = 1,~~\Delta f_1 (e_1,  e_2) = f_1(0) = 0,
$$
$$
	\Delta f_1 (e_2,  e_1) = f_1(e_2) = 0,~~\Delta f_1 (e_2,  e_2) = f_1(0) = 0.
	$$
	Formulas for $f_2$, 
	$$
	\Delta f_2 (e_1,  e_1) = f_2(e_1) = 0,~~\Delta f_2 (e_1,  e_2) = f_2(0) = 0,
	$$
	$$
	\Delta f_2 (e_2,  e_1) = f_2(e_2) = 1,~~\Delta f_2 (e_2,  e_2) = f_2(0) = 0.
	$$
	Hence, the  comultiplication has the form
	$$
	\Delta f_1 = f_1 \otimes f_1,~~\Delta f_2 =  f_2 \otimes f_1.
	$$
	
	2) Multiplication 
	$$
	e_1^2 = e_1,~~e_1 e_2 = e_2,~~e_2 e_1 = 0,~~e_2^2 =0.
	$$
	In this case we have the formulas for $f_1$,
	$$
	\Delta f_1 (e_1,  e_1) = f_1(e_1) = 1,~~\Delta f_1 (e_1,  e_2) = f_1(e_2) = 0,
	$$
	$$
	\Delta f_1 (e_2,  e_1) = f_1(0) = 0,~~\Delta f_1 (e_2,  e_2) = f_1(0) = 0.
	$$
	Formulas for $f_2$, 
	$$
	\Delta f_2 (e_1,  e_1) = f_2(e_1) = 0,~~\Delta f_2 (e_1,  e_2) = f_2(e_2) = 1,
	$$
	$$
	\Delta f_2 (e_2,  e_1) = f_2(0) = 0,~~\Delta f_2 (e_2,  e_2) = f_2(0) = 0.
	$$
	Hence, the  comultiplication has the form
	$$
	\Delta f_1 = f_1 \otimes f_1,~~\Delta f_2 =  f_1 \otimes f_2.
	$$
	
	3) Multiplication 
	$$
	e_1^2 = e_1,~~e_1 e_2 = 0,~~e_2 e_1 = 0,~~e_2^2 =0.
	$$
	In this case we have the formulas for $f_1$,
	$$
	\Delta f_1 (e_1,  e_1) = f_1(e_1) = 1,~~\Delta f_1 (e_1,  e_2) = f_1(0) = 0,
	$$
	$$
	\Delta f_1 (e_2,  e_1) = f_1(0) = 0,~~\Delta f_1 (e_2,  e_2) = f_1(0) = 0.
	$$
	Formulas for $f_2$, 
	$$
	\Delta f_2 (e_1,  e_1) = f_2(e_1) = 0,~~\Delta f_2 (e_1,  e_2) = f_2(0) = 0,
	$$
	$$
	\Delta f_2 (e_2,  e_1) = f_2(0) = 0,~~\Delta f_2 (e_2,  e_2) = f_2(0) = 0.
	$$
	Hence, the  comultiplication has the form
	$$
	\Delta f_1 = f_1 \otimes f_1,~~\Delta f_2 = 0.
	$$
	
	4) Multiplication 
	$$
	e_1^2 = e_2,~~e_1 e_2 = 0,~~e_2 e_1 = 0,~~e_2^2 =0.
	$$
	In this case we have the formulas for $f_1$,
	$$
	\Delta f_1 (e_1,  e_1) = f_1(e_2) = 0,~~\Delta f_1 (e_1,  e_2) = f_1(0) = 0,
	$$
	$$
	\Delta f_1 (e_2,  e_1) = f_1(0) = 0,~~\Delta f_1 (e_2,  e_2) = f_1(0) = 0.
	$$
	Formulas for $f_2$, 
	$$
	\Delta f_2 (e_1,  e_1) = f_2(e_2) = 1,~~\Delta f_2 (e_1,  e_2) = f_2(0) = 0,
	$$
	$$
	\Delta f_2 (e_2,  e_1) = f_2(0) = 0,~~\Delta f_2 (e_2,  e_2) = f_2(0) = 0.
	$$
	Hence, the  comultiplication has the form
	$$
	\Delta f_1 = 0,~~\Delta f_2 = f_2 \otimes f_2.
	$$
	
	5) In this case we have zero multiplication. Hence, comultiplication also will be zero.
	
	We proved the following proposition.

	\begin{proposition}
		All non-counital coassociative comultiplications on a 2-dimensional space can be brought to one of the following types by a linear change of variables ($\{ f_1, f_2 \}$ is an appropriate basis):
		
		1) $\Delta f_1 = f_1 \otimes f_1,~~\Delta f_2 =  f_1 \otimes f_2;$
		
		2) $\Delta f_1 = f_1 \otimes f_1,~~\Delta f_2 =  f_2 \otimes f_1;$
		
		3) $\Delta f_1 = f_1 \otimes f_1,~~\Delta f_2 = 0$;
		
		4) $\Delta f_1 = 0,~~\Delta f_2 = f_2 \otimes f_2.$
		
		5) $\Delta f_1 = 0,~~\Delta f_2 = 0.$
		
	\end{proposition}
	
		\subsection{Self-distributive bialgebras with comultiplication 1) and 2)}
	
	In this subsection we are considering a bialgebra $A$ which has a basis $x, y$, comultiplication is defined by formulas 1),
	$$
	\Delta x = x \otimes x,~~\Delta y = x \otimes y.
	$$
	We want to define multiplication
	$$
	x^2 = a_1 x + a_2 y,~~x y = b_1 x + b_2 y,~~y x = c_1 x + c_2 y,~~y^2= d_1 x + d_2 y,
	$$
	such that the coproduct is a homomorphism of algebra $A$, i.~e. for any $a, b \in A$ holds
	\begin{equation} \label{gom}
		\Delta (a b) = \Delta(a) \Delta(b),
	\end{equation}
	and this bialgebra is self-distributive that means that for any $a, b \in A$ holds 
	\begin{equation} \label{SD}
		(a b) c = (a c^{(1)}) (b c^{(2)}).
	\end{equation}
	
	Let us check condition (\ref{gom}) for basis elements $x$ and $y$.
	
	1) The equality $\Delta (x x) = \Delta(x) \Delta(x)$ is equivalent to the equality
	$$
	a_1 x \otimes x + 	a_2 x \otimes y = (x \otimes x) (x \otimes x) \Leftrightarrow  a_1 x \otimes x 	a_2 x \otimes y = (a_1 x + a_2 y) \otimes (a_1 x + a_2 y).
	$$
	Transforming the right hand side, we get
	$$
	a_1 x \otimes x + a_2 x \otimes y = a_1^2  x \otimes x + a_1 a_2 x \otimes y + a_1 a_2 y \otimes  x + a_2^2 y \otimes y.
	$$
	This equality holds if and only if the following system has solutions in $\Bbbk$,
	$$
	\begin{cases}
		a_1^2 =a_1, \\
		a_1 a_2 = 0, \\
	a_2 = a_1 a_2, \\
		a_2^2 = 0.
	\end{cases}
	$$
	Hence, $a_1 = 0$, or $a_1 = 1$, in both cases $a_2 = 0$.
	
	2) The equality $\Delta (y y) = \Delta(y) \Delta(y)$ is equivalent to the equality
	$$
d_1 x \otimes x + d_2 x \otimes y = a_1 d_1  x \otimes x + a_1 d_2 x \otimes y + a_2 d_1 y \otimes  x + a_2 d_2 y \otimes y.
	$$
	This equality holds if and only if the following system has solutions in $\Bbbk$,
	$$
	\begin{cases}
		a_1 d_1 =d_1, \\
		a_1 d_2 = d_2, \\
			a_2 d_1 = 0.\\
a_2 d_2 = 0. 
	\end{cases}
	$$
Since $a_2 = 0$, we see that $a_1 = 1$ and $d_1, d_2$ are arbitrary,  or $a_1 \not= 1$ and $d_1 = d_2 = 0$.

Comparing with the case 1) we get two possible multiplications

M1: $	x^2 = 0,~~x y = b_1 x + b_2 y,~~y x = c_1 x + c_2 y,~~y^2= 0,$

M2: $	x^2 = x,~~x y = b_1 x + b_2 y,~~y x = c_1 x + c_2 y,~~y^2= d_1 x + d_2 y.$
	
{\it Multiplication M1}.  Consider the equality $\Delta (x y) = \Delta(x) \Delta(y)$ is equivalent to the equality
	$$
b_1 x \otimes x + b_2 x \otimes y = 0.
	$$
Hence, $b_1 = b_2 = 0$.

Consider the equality $\Delta (y x) = \Delta(y) \Delta(x)$ is equivalent to the equality
	$$
c_1 x \otimes x + c_2 x \otimes y = 0.
	$$
Hence, $c_1 = c_2 = 0$.

We see, that the multiplication M1 gives an algebra with zero multiplication.

{\it Multiplication M2}.  Consider the equality $\Delta (x y) = \Delta(x) \Delta(y)$ is equivalent to the equality
	$$
b_1 x \otimes x + b_2 x \otimes y = b_1 x \otimes x + b_2 x \otimes y,
	$$
which is true for arbitrary  $b_1$ and  $b_2$.

Consider the equality $\Delta (y x) = \Delta(y) \Delta(x)$ is equivalent to the equality
	$$
c_1 x \otimes x + c_2 x \otimes y = c_1 x \otimes x + c_2 x \otimes y,
	$$	
which is true for arbitrary  $c_1$ and  $c_2$.	

Consider the equality $\Delta (y y) = \Delta(y) \Delta(y)$ is equivalent to the equality
	$$
d_1 x \otimes x + d_2 x \otimes y = d_1 x \otimes x + d_2 x \otimes y,
	$$	
which is true for arbitrary  $d_1$ and  $d_2$.

	Further, we analyze self-distributivity axiom (\ref{SD}). It is enough to check it only for basis elements.

1) If $a = b = c = x$, then $(x x) x = (x x^{(1)}) (x x^{(2)})$. Since $x^{(1)} = x^{(2)} = x$, we  get $x = x$.
	
2) If $a = x,  b = y,  c = x$, then $(x y) x = (x x ^{(1)}) (y x^{(2)})$. Since $x^{(1)} = x$, $x^{(2)} = x$, we  get the equality
$$
(b_1 + b_2 c_1) x + (b_2 c_2) y =  (c_1  + c_2 b_1 ) x + (c_2  b_2) y,
$$	
which equivalent to the equality:
	$b_1 +b_2 c_1 = c_1+ c_2b_1, \\$
So, $$ b_{1}=\frac{c_{1}\left( 1-b_{2}\right)  }{1-c_2}, c_{2}\neq 1 .$$	

3) If $a = x,  b = x,  c = y$, then $(x x) y = (x y ^{(1)}) (x y^{(2)})$. Since $y^{(1)} = x$, $y^{(2)} = y$, we  get the equality
$$
b_1  x + b_2 y =  (b_1  +  b_1 b_2 ) x + b_2^2 y,
$$	
which equivalent to the system:
$$
	\begin{cases}
b_1 b_2 = 0, \\
b_1  =  b_2^2. 
	\end{cases}
$$

If $b_1 = 0$, then $b_2=0, c_1=0, c_{2}\neq 1$.

If $b_1 = 0$ and  $b_2=1$, then $c_{2}\neq 1$,  $b_2$ can be arbitrary.

If $b_2 = 0$, then    $b_{1}=\frac{c_{1}} {1-c_2} $,  $c_{2}\neq 1$,  $c_1$ can be arbitrary.

Analyzing other possibilities, we get

	\begin{proposition} \label{p34}
		Let $A$ be a vector space with a basis $x, y$, on which is defined a comultiplication
		$$
		\Delta x =  x \otimes x,~~\Delta y = x \otimes y.
		$$
		Then the multiplication 
		
a) $~x^2 =0,~~xy = 0,~~y x = 0,~~y^2 = 0$;

b)		$~x^2 =x,~~xy = 0,~~y x = c y,~~y^2 = 0,~~c \in \Bbbk$;

c)		$~x^2 =x,~~xy = y,~~y x = 0,~~y^2 = 0$;

d)		$~x^2 =x,~~xy = y,~~y x = c x,~~y^2 = c y,~~c \in \Bbbk;$

e)		$~x^2 =x,~~xy = y,~~y x = c x -y,~~y^2 = d x,~~c, d\in \Bbbk$;

f)		$~x^2 =x,~~xy = y,~~y x = -y,~~y^2 = d x,~~ d  \in \Bbbk$;

g)		$~x^2 =x,~~xy = y,~~y x = c y,~~y^2 = 0,~~c  \in \Bbbk$;

h)		$~x^2 =x,~~xy = c x,~~y x = c x,~~y^2 = c^2 x,~~c_1 \in \Bbbk$;

k)		$~x^2 =x,~~xy = \frac{c}{2} x,~~y x = c x-y,~~y^2 = \frac{c^2}{2} x -  \frac{c}{2} y,~~c \in \Bbbk$;

l)		$~x^2 =x,~~xy = y,~~y x = c_1x+c_2y,~~y^2 = \frac{-c^{2}_{1}c_{2}}{\left( 1-c_2\right)^{2}  }  x + \frac{c_{1}\left( c_{2}+1\right)  }{1-c_{2}}  y ,~~c_1,c_2 \in \Bbbk , c_{2}\neq 1$.

\noindent
		gives a non-counital self-distributive bialgebra. 
	\end{proposition}
	
	\medskip

Now we are considering a bialgebra $A$ which has a basis $x, y$, comultiplication is defined by formulas 2),
	$$
	\Delta x = x \otimes x,~~\Delta y = y \otimes x.
	$$
	We want to define multiplication
	$$
	x^2 = a_1 x + a_2 y,~~x y = b_1 x + b_2 y,~~y x = c_1 x + c_2 y,~~y^2= d_1 x + d_2 y,
	$$
	such that the coproduct is a homomorphism of algebra $A$, i.~e. for any $a, b \in A$ holds
	\begin{equation} \label{gom}
		\Delta (a b) = \Delta(a) \Delta(b),
	\end{equation}
	and this bialgebra is self-distributive that means that for any $a, b \in A$ holds 
	\begin{equation} \label{SD}
		(a b) c = (a c^{(1)}) (b c^{(2)}).
	\end{equation}
	
	Let us check condition (\ref{gom}) for basis elements $x$ and $y$.
	
	1) The equality $\Delta (x x) = \Delta(x) \Delta(x)$ is equivalent to the equality
	$$
	a_1 x \otimes x + 	a_2 x \otimes y = (x \otimes x) (x \otimes x) \Leftrightarrow  a_1 x \otimes x 	a_2 x \otimes y = (a_1 x + a_2 y) \otimes (a_1 x + a_2 y).
	$$
	Transforming the right hand side, we get
	$$
	a_1 x \otimes x + a_2 x \otimes y = a_1^2  x \otimes x + a_1 a_2 x \otimes y + a_1 a_2 y \otimes  x + a_2^2 y \otimes y.
	$$
	This equality holds if and only if the following system has solutions in $\Bbbk$,
	$$
	\begin{cases}
		a_1^2 =a_1, \\
		a_1 a_2 = 0, \\
	a_2 = a_1 a_2, \\
		a_2^2 = 0.
	\end{cases}
	$$
	Hence, $a_1 = 0$, or $a_1 = 1$, in both cases $a_2 = 0$.
	
	2) The equality $\Delta (y y) = \Delta(y) \Delta(y)$ is equivalent to the equality
	$$
d_1 x \otimes x + d_2 y \otimes x = a_1 d_1  x \otimes x + a_2 d_1 x \otimes y + a_1 d_2 y \otimes  x + a_2 d_2 y \otimes y.
	$$
	This equality holds if and only if the following system has solutions in $\Bbbk$,
	$$
	\begin{cases}
		a_1 d_1 =d_1, \\
		a_1 d_2 = d_2, \\
			a_2 d_1 = 0.\\
a_2 d_2 = 0. 
	\end{cases}
	$$
Since $a_2 = 0$, we see that $a_1 = 1$ and $d_1, d_2$ are arbitrary,  or $a_1 \not= 1$ and $d_1 = d_2 = 0$.

Comparing with the case 1) we get two possible multiplications

M1: $	x^2 = 0,~~x y = b_1 x + b_2 y,~~y x = c_1 x + c_2 y,~~y^2= 0,$

M2: $	x^2 = x,~~x y = b_1 x + b_2 y,~~y x = c_1 x + c_2 y,~~y^2= d_1 x + d_2 y.$
	
{\it Multiplication M1}.  Consider the equality $\Delta (x y) = \Delta(x) \Delta(y)$ is equivalent to the equality
	$$
b_1 x \otimes x + b_2  y \otimes x = 0.
	$$
Hence, $b_1 = b_2 = 0$.

Consider the equality $\Delta (y x) = \Delta(y) \Delta(x)$ is equivalent to the equality
	$$
c_1 x \otimes x + c_2 y \otimes x = 0.
	$$
Hence, $c_1 = c_2 = 0$.

We see, that the multiplication M1 gives an algebra with zero multiplication.

{\it Multiplication M2}.  Consider the equality $\Delta (x y) = \Delta(x) \Delta(y)$ is equivalent to the equality
	$$
b_1 x \otimes x + b_2 y \otimes x = b_1 x \otimes x + b_2 y \otimes x,
	$$
which is true for arbitrary  $b_1$ and  $b_2$.

Consider the equality $\Delta (y x) = \Delta(y) \Delta(x)$ is equivalent to the equality
	$$
c_1 x \otimes x + c_2 y \otimes x = c_1 x \otimes x + c_2 y \otimes x,
	$$	
which is true for arbitrary  $c_1$ and  $c_2$.	

The equality $\Delta (y y) = \Delta(y) \Delta(y)$ is equivalent to the equality
	$$
d_1 x \otimes x + d_2 y \otimes x = d_1 x \otimes x + d_2 y \otimes x,
	$$	
which is true for arbitrary  $d_1$ and  $d_2$.

	Further, we analyze self-distributivity axiom (\ref{SD}). It is enough to check it only for basis elements.

1) If $a = b = c = x$, then $(x x) x = (x x^{(1)}) (x x^{(2)})$. Since $x^{(1)} = x^{(2)} = x$, we  get $x = x$.
	
2) If $a = x,  b = y,  c = x$, then $(x y) x = (x x ^{(1)}) (y x^{(2)})$. Since $x^{(1)} = x$, $x^{(2)} = x$, we  get the equality
$$
(b_1 + b_2 c_1) x + (b_2 c_2) y =  (c_1  + c_2 b_1 ) x + (c_2  b_2) y,
$$	
which equivalent to the equality:
	$b_1 +b_2 c_1 = c_1+ c_2b_1, \\$
So, $$ b_{1}=\frac{c_{1}\left( 1-b_{2}\right)  }{1-c_2}, c_{2}\neq 1 .$$	

3) If $a = x,  b = x,  c = y$, then $(x x) y = (x y ^{(1)}) (x y^{(2)})$. Since $y^{(1)} = y$, $y^{(2)} = x$, we  get the equality
$$
b_1  x + b_2 y =  (b_1  +  b_2 c_1 ) x + b_2 c_2 y,
$$	
which equivalent to the system:
$$
	\begin{cases}
b_2 c_1 = 0, \\
b_2  = b_2 c_2. 
	\end{cases}
$$

If $b_2 = 0$, then $b_1=\frac{c_{1}}{1-c_{2}} ,  c_{2}\neq 1$, $c_1$ can be arbitrary.

If $c_1 = 0$, then    $b_{1}=0$,  $c_{2}\neq 1$,  $b_2$ can be arbitrary.

Analyzing other possibilities, we get

\begin{proposition} \label{c42}
		Let $A$ be a vector space with a basis $x, y$, on which is defined a comultiplication
		$$
\Delta x =  x \otimes x,~~\Delta y = y \otimes x.
		$$
		Then the multiplication 

a)		$~x^2 =0,~~xy = 0,~~y x = 0,~~y^2 = 0$;

b)		$~x^2 =x,~~xy = 0,~~y x = cy,~~y^2 = 0,~~c\in \Bbbk$;

c)		$~x^2 =x,~~xy = c x,~~y x = cx,~~y^2 = c^2x,~~c \in \Bbbk$;

d)		$~x^2 =x,~~xy = b y,~~y x = c y,~~y^2 =0 ,~~b, c \in \Bbbk$;

e)		$~x^2 =x,~~xy =  \frac{c_{1}  }{1-c_{2}} x,~~y x = c_1x+c_2y,~~y^2 = \frac{c^{2}_{1}}{1-c_{2}  }  x + \frac{c_{1} c_{2} }{1-c_{2}}  y ,~~~c_1, c_2 \in \Bbbk, c_{2}\neq 1$;
		gives a non-counital self-distributive bialgebra. 
	\end{proposition}

	\subsection{Self-distributive bialgebras with comultiplication 3) and 4)}
	
	In this subsection we are considering a bialgebra $A$ which has a basis $x, y$, comultiplication is defined by formulas 3),
	$$
	\Delta x = x \otimes x,~~\Delta y = 0.
	$$
	We want to define multiplication
	$$
	x x = a_1 x + a_2 y,~~x y = b_1 x + b_2 y,~~y x = c_1 x + c_2 y,~~yy = d_1 x + d_2 y,
	$$
	such that the coproduct is a homomorphism of algebra $A$, i.~e. for any $a, b \in A$ holds
	\begin{equation} \label{Sog}
		\Delta (a b) = \Delta(a) \Delta(b),
	\end{equation}
	and this bialgebra is self-distributive that means that for any $a, b \in A$ holds 
	\begin{equation} \label{SD1}
		(a b) c = (a c^{(1)}) (b c^{(2)}).
	\end{equation}
	
	Let us check condition (\ref{Sog}) for basis elements $x$ and $y$.
	
	1) The equality $\Delta (x x) = \Delta(x) \Delta(x)$ is equivalent to the equality
	$$
	d_1 x \otimes x = (x \otimes x) (x \otimes x) \Leftrightarrow  a_1 x \otimes x = (a_1 x + a_2 y) \otimes (a_1 x + a_2 y).
	$$
	Transforming the right hand side, we get
	$$
	a_1 x \otimes x = a_1^2  x \otimes x + a_1 a_2 x \otimes y + a_1 a_2 y \otimes  x + a_2^2 y \otimes y.
	$$
	This equality holds if and only if the following system has solutions in $\Bbbk$,
	$$
	\begin{cases}
		a_1^2 =a_1, \\
		a_1 a_2 = 0, \\
		a_2^2 = 0.
	\end{cases}
	$$
	Hence, $a_1 = 0$, or $a_1 = 1$, in both cases $a_2 = 0$.
	
	2) The equality $\Delta (y y) = \Delta(y) \Delta(y)$ is equivalent to the equality
	$$
	d_1 x \otimes x = 0.
	$$
	
	Hence, $d_1 = 0$.
	
	Analyzing other cases we do not get new restrictions on the multiplication. Hence, we have two multiplications:
	$$
	x^2 = 0,~\mbox{or}~x^2 =x,~~xy = b_2 y,~~y x = c_2 y,~~y^2 = 0,~~b_2, c_2 \in \Bbbk.
	$$
	
	Further, we analyze self-distributivity axiom. It is enough to check it only for basis elements.
	
	{\it Case 1}:  $~x^2 =0,~~xy = b_2 y,~~y x = c_2 y,~~y^2 = 0,~~b_2, c_2 \in \Bbbk.
	$
	
	1) If $a = b = c = y$, then $(y y) y = (y y^{(1)}) (y y^{(2)})$. Since $y^{(1)} = y^{(2)} = 0$, we  get $d_2 = 0$.
	
	2) If $a = y$, $b = c = y$, then $(y x) x = (y x^{(1)}) (x x^{(2)})$. Since $x^{(1)} = x^{(2)} = x$, we  get $c_2 = 0$.
	
	Analyzing other possibilities,  cases we do not get new restrictions on the multiplication. Hence, the  multiplications:
	$$
	~x^2 =0,~~xy = b_2 y,~~y x = 0,~~y^2 = 0,~~b_2 \in \Bbbk
	$$
	gives a self-distributive bialgebra.
	
	{\it Case 2}:  $~x^2 =x,~~xy = b_2 y,~~y x = c_2 y,~~y^2 = 0,~~b_2, c_2 \in \Bbbk.
	$
	
	1) If $a = b = c = y$, then $(y y) y = (y y^{(1)}) (y y^{(1)})$. Since $y^{(1)} = y^{(2)} = 0$, we  get $d_2 = 0$.
	
	2) If $a = b = x$, $c = y$, then $(x x) y = (x y^{(1)}) (x y^{(1)})$ and we  get $b_2 = 0$.
	
	Analyzing other possibilities,  cases we do not get new restrictions on the multiplication. Hence, the  multiplications:
	$$
	~x^2 =x,~~xy = 0,~~y x = c_2 y,~~y^2 = 0,~~b_2 \in \Bbbk
	$$
	gives a self-distributive bialgebra.
	
	\begin{proposition} \label{c3}
		Let $A$ be a vector space with a basis $x, y$, on which is defined a comultiplicatiom
		$$
		\Delta x = x \otimes x,~~\Delta y = 0.
		$$
		Then the multiplication 
		$$
		~x^2 =0,~~xy = b_2 y,~~y x = 0,~~y^2 = 0,~~b_2 \in \Bbbk,
		$$
		or
		$$
		~x^2 =x,~~xy = 0,~~y x = c_2 y,~~y^2 = 0,~~c_2 \in \Bbbk
		$$
		gives a non-counital self-distributive bialgebra. 
	\end{proposition}
	
	\medskip  
	
	If we are considering a vector space  $A$ which has a basis $x, y$ and  comultiplication is defined by the formulas 4),
	$$
	\Delta x = 0,~~\Delta y = y \otimes y,
	$$ 
	then we get it if we permute $x$ and $y$. Hence,  from Proposition \ref{c3} follows
	
	\begin{proposition} \label{c4}
		Let $A$ be a vector space with a basis $x, y$, on which is defined a comultiplicatiom
		$$
		\Delta x = 0,~~\Delta y = y \otimes y.
		$$
		Then the multiplication 
		$$
		~x^2 =0,~~xy = 0,~~y x = b x,~~y^2 = 0,~~b \in \Bbbk,
		$$
		or
		$$
		~x^2 =0,~~xy = c x,~~y x = 0,~~y^2 = y,~~b \in \Bbbk
		$$
		gives a non-counital self-distributive bialgebra. 
	\end{proposition}

	\subsection{Self-distributive bialgebras with comultiplication 5)}
	
	In this subsection we are considering a bialgebra $A$ which has a basis $x, y$, comultiplication is defined by formulas 5),
	$$
	\Delta x = 0,~~\Delta y = 0.
	$$
	It is evident that for any multiplication  the coproduct is a homomorphism of algebra $A$, i.~e. for any $a, b \in A$ holds
	\begin{equation} \label{Sog1}
		\Delta (a b) = \Delta(a) \Delta(b).
	\end{equation}
	
	We want to define a multiplication
	$$
	x x = a_1 x + a_2 y,~~x y = b_1 x + b_2 y,~~y x = c_1 x + c_2 y,~~yy = d_1 x + d_2 y,
	$$
	such that this bialgebra satisfies the self-distributivity axiom, 
	$$
	(a b) c = (a c^{(1)}) (b c^{(2)}),~~a, b \in A,
	$$
	which is equivalent to
	\begin{equation} \label{SD3}
		(a b) c = 0,~~a, b \in A.
	\end{equation}
	
	1) If $a = b = c = x$, then
	$$
	(x x) x = (a_1^2 + a_2 c_1) x + (a_1 a_2 + a_2 c_2) y = 0 
	$$
	and we have a system
	$$
	\begin{cases}
		a_1^2 + a_2 c_1 = 0, \\
		a_2 (a_1 +  c_2) = 0.
	\end{cases}
	$$
	If $a_2 = 0$, then $a_1 = 0$ and $c_1$ and $c_2$ can be arbitrary.
	
	If $a_2 \not= 0$, then $c_1 = -a_1^2 / a_2$ and $c_2 = -a_1$.
	
	2) If $a = b = c = y$, then
	$$
	(y y) y = (d_1 b_1 + d_1 d_2) x + (b_2 d_1 + d_2^2) y = 0 
	$$
	and we have a system
	$$
	\begin{cases}
		d_1 (b_1 + d_2) = 0, \\
		b_2 d_1 + d_2^2 = 0.
	\end{cases}
	$$
	If $d_1 = 0$, then $d_2 = 0$ and $b_1$ and $b_2$ can be arbitrary.
	
	If $d_1 \not= 0$, then $b_1 = -d_2$ and $b_2 = -d_2^2 / d_1$.

Hence, $x^2 x = y^2 y = 0$ for one the following four multiplications:

M1. $	x^2 = 0,~~x y = b_1 x + b_2 y,~~y x = c_1 x + c_2 y,~~y^2 = 0,$

M2. $	x^2 = 0,~~x y = - d_2 x - \frac{d_2^2}{d_1} y,~~y x = c_1 x + c_2 y,~~y^2 = d_1 x + d_2 y,~~d_1 \not= 0,$ 

M3. $x^2 = a_1 x + a_2 y,~~ x y = b_1 x + b_2 y,~~y x = - \frac{a_1^2}{a_2} x - a_1 y,~~y^2 = 0,~~a_2 \not= 0,$

M4. $x^2 = a_1 x + a_2 y,~~x y = - d_2 x - \frac{d_2^2}{d_1} y,~~y x = - \frac{a_1^2}{a_2} x - a_1 y,~~y^2 = d_1 x + d_2 y,~~$ $d_1 \not= 0, ~~$ $a_2 \not= 0.$

Further we will consider each of these multiplications and find conditions under which $A^2 A = 0$.

{\it Multiplication M1}:

M1. 3) If $a = b =x$, $c = y$, then
	$$
	(x x) y = 0.
	$$

M1. 4) If $a = c =x$, $b = y$, then
	$$
	(x y) x  = b_2 (c_1  x +  c_2 y) = 0. 
	$$
It is possible if $b_2 = 0$ and $c_1, c_2$ are arbitrary, or $b_2 \not = 0$ and $c_1 = c_2 = 0$.

M1. 5) If $a = y$, $b = c =x$,  then
	$$
	(y x) x  =  c_2 (c_1 x + c_2 y) = 0. 
	$$
It is possible, if $c_2 = 0$ and $c_1$ is arbitrary.

M1. 6) If $a = b =y$, $c = x$, then $(y y) x   = 0.$

M1. 7) If $a = c =y$, $b = x$, then
	$$
	(y x) y  = c_1 (b_1 x + b_2 y) = 0. 
	$$

M1. 8) If $a = x$, $b = c =y$,  then
	$$
	(x y) y  =  b_1 (b_1 x + b_2 y) = 0. 
	$$
Hence, $b_1 = 0$. Since $c_2 = 0$, we have equation $b_2 c_1 = 0$. It means that we have two multiplications:

M1a). $	x^2 = 0,~~x y = 0,~~y x = c_1 x,~~y^2 = 0,$

M1b). $	x^2 = 0,~~x y = b_2 y,~~y x =0,~~y^2 = 0.$

\medskip

{\it Multiplication M2}:

M2. 3) If $a = b =x$, $c = y$, then
	$$
	(x x) y = 0.
	$$
	
M2. 4) If $a = c =x$, $b = y$, then
	$$
	(x y) x  = - \frac{d_2^2}{d_1} (c_1  x +  c_2 y) = 0. 
	$$
It means that $d_2 = 0$, or $c_1 = c_2 = 0$.

M2. 5) If $a = y$, $b = c =x$,  then
	$$
	(y x) x  =  c_2 (c_1 x + c_2 y) = 0. 
	$$
It is possible, if $c_2 = 0$ and $c_1$ is arbitrary.

M2. 6) If $a = b =y$, $c = x$, then 
$$
(y y) x   = d_2 (c_1 x + c_2 y) = 0.
$$
It is possible, if $d_2 = 0$ or $c_1 = c_2 = 0$.
	
M2. 7) If $a = c =y$, $b = x$, then
	$$
	(y x) y  = c_1 (- d_2 x - \frac{d_2^2}{d_1} y) + c_2 (d_1 x + d_2 y) = 0. 
	$$
Since $c_2 = 0$, then $c_1 = 0$ or $d_2 = 0$.
	
M2. 8) If $a = x$, $b = c =y$,  then $	(x y) y  = 0$ and we do not have new conditions on the coefficients.
	
We get that $c_2 = 0$ and $d_2 = 0$ or $c_1 = 0$. 
 It means that we have two multiplications:\\

M2a). $	x^2 = 0,~~x y = 0,~~y x = c_1 x,~~y^2 = d_1 x,$\\

M2b). $	x^2 = 0,~~x y = - d_2 x - \frac{d_2^2}{d_1} y,~~y x =0,~~y^2 = d_1 x + d_2 y.$

\medskip

{\it Multiplication M3}:

M3. 3) If $a = b =x$, $c = y$, then
	$$
	(x x) y = a_1 (b_1 x + b_2y) = 0.
	$$
Hence, we have two multiplications:\\

M3a). $	x^2 = a_2 x,~~x y = b_1 x + b_2 y,~~y x = 0,~~y^2 = 0,~~~a_2 \not= 0$\\

M3b). $	x^2 = a_1 x + a_2 y,~~x y = 0,~~yx = - a_1 \left(\frac{a_1}{a_2} x + y \right),~~y^2 = 0,~~~a_2 \not= 0.$

\medskip

{\it Multiplication M4}:
	
Using calculations which are similar to multiplications M1--M3, we get an algebra with multiplication 

M4. $	x^2 = a_1 \left( x + \frac{d_2}{d_1} y \right),~~x y = - d_2 \left( x + \frac{d_2}{d_1} y \right),~~yx = - \frac{a_1 d_1}{d_2} \left( x + \frac{d_2}{d_1} y \right),$

$y^2 =  d_1 \left( x + \frac{d_2}{d_1} y \right),~~~d_1 \not= 0.$
	
	\begin{proposition} \label{c5}
		Let $A$ be a vector space with a basis $x, y$, on which is defined a comultiplicatiom
		$$
		\Delta x = 0,~~\Delta y = 0.
		$$
		Then the following multiplications 
		
		a) $x^2 = 0,~~x y = 0,~~y x = c x,~~y^2 = 0,~~ c \in \Bbbk$; 
		
		b) $	x^2 = 0,~~x y = b y,~~y x =0,~~y^2 = 0,~~b \in \Bbbk$;
		
		c) $	x^2 = 0,~~x y = 0,~~y x = c x,~~y^2 = d x,~~c, d  \in \Bbbk, d \not = 0$;
		
        d)	$	x^2 = 0,~~x y = - d_2 x - \frac{d_2^2}{d_1} y,~~y x =0,~~y^2 = d_1 x + d_2 y ~~d_1, d_2  \in \Bbbk, d_1 \not = 0$	;
        
        e) $	x^2 = a x,~~x y = b_1 x + b_2 y,~~y x = 0,~~y^2 = 0,~~~a \not= 0$;

        f) $	x^2 = a_1 x + a_2 y,~~x y = 0,~~yx = - a_1 \left(\frac{a_1}{a_2} x + y \right),~~y^2 = 0,~~~a_2 \not= 0$;
        
       g) $	x^2 = a_1 \left( x + \frac{d_2}{d_1} y \right),~~x y = - d_2 \left( x + \frac{d_2}{d_1} y \right),~~yx = - \frac{a_1 d_1}{d_2} \left( x + \frac{d_2}{d_1} y \right),$

$y^2 =  d_1 \left( x + \frac{d_2}{d_1} y \right),~~~d_1 \not= 0.$

\noindent
		gives a non-counital self-distributive bialgebra. 
	\end{proposition}
	
	\begin{remark}
		A self-distributive bialgebra with trivial comultiplication satisfies the identity $A^2 A = 0$ and hence is a self distributive algebra. 
	\end{remark}

		\bigskip
	
	%%%%%%%%%%%%%%%%%%%%%%%%%%%%%%%%%%%%%%%%%%%%%%%%
	
	\section{Idempotents  and quandles in 2-dimensional algebras} \label{id}

As we remarked in the introduction, it is possible to use quandle rings for construction of knot invariants.   
It was  proven  in \cite{ES} that quandle rings and their sets
of idempotents lead to  proper enhancements of the well-known quandle coloring
invariant of links. In this section we are studying the following question: what idempotents and quandles lie in these   algebras?

	In propositions \ref{p34} - \ref{c5} we constructed 2-dimensional non-counital bialgebras.  More of these algebras contain only two idempotents: $\{ 0, x \}$. Some interesting results  gives 
	
\begin{proposition}
a) Algebra with multiplication
$$
x^2 =x,~~xy = 0,~~y x = y,~~y^2 = 0
$$
contains trivial quandle $Q_1 = \{ x + \alpha y ~|~\alpha \in \Bbbk \}$. 

b) Algebra with multiplication
$$
x^2 =x,~~xy =  y,~~y x = c x,~~y^2 = c y,~~c \in \Bbbk
$$
contains a trivial  quandle
$$
(\{ (1 - \alpha c) x + \alpha y ~|~\alpha \in \Bbbk \}; \, \cdot^{op} )
$$
where $a \cdot^{op} b = b \cdot a$ and `$\cdot$' is the usual multiplication in algebra.

%c) Algebra with multiplication
%$$
%x^2 =x,~~xy = \frac{c}{2} x,~~y x = c x,~~y^2 = \frac{c^2}{2} x -  \frac{c}{2} y,~~c \in \Bbbk
%$$
%contains trivial 2-element quandle
%$$
%\{ (2 \pm \sqrt{2}) x - \frac{2}{c} y \}.
%$$

c) Algebra with multiplication
$$
x^2 =x,~~xy = y,~~y x = cx,~~y^2 =  c y ,~~c_1, c \in \Bbbk,
$$
contains quandle $Q_2 = \{ (1 - c \alpha) x + \alpha y ~|~\alpha \in \Bbbk \}$. 

\end{proposition}

\begin{proof}
a) This algebra comes from Proposition \ref{p34} b) if we put $c=1$. It is easy to see that 
for any $\alpha, \beta \in \Bbbk$ holds $( x + \alpha y)( x + \beta y) =  x + \alpha y$.

b) This algebra comes from Proposition \ref{p34} b) and assertion follows from the equality
$$
((1 - \alpha c) x + \alpha y ) \cdot ((1 - \beta  c) x + \beta  y)  =(1 - \beta  c) x + \beta  y, ~~\alpha, \beta \in \Bbbk.
$$

%c) This algebra is algebra from  Proposition  \ref{p34} k).

c) This algebra is algebra from  Proposition  \ref{p34} l) if we put $c_2 = 0$.
\end{proof}

	Let us consider the algebra in Proposition \ref{c5}. 
	
	In \cite{BKT1} we proved that if $\Bbbk$ contains more than 3 elements, then an algebra $A$ over $\Bbbk$ be self-distributive, the  identity $A^2 A = 0$ must be true.
This means that for any $u, v, w \in A$ holds
$$
(u v) w = 0.
$$

\begin{proposition}
Let $\Bbbk$ be a field with more than 3 elements, $A$ be a non-zero self-distributive algebra. Then any idempotent in $A$ is zero and $A$ can not be a rack.
\end{proposition}

\begin{proof}
As we know in this situation for $A$ holds $A^2A = 0$. Suppose that $u \in A$ is an idempotent, i.e. $u^2 = u$. After multiplication on $u$ from the right, we get
$$
u^2 u = u^2.
$$
Hence, $u^2=0$, but it means that $u = 0$.

Suppose that $A$ is a rack. It means that for any $u, v \in A$ there exists unique $z \in A$ such that $z u = v$. But for any $w \in A$ we have $(z u) w = v w = 0$. Hence, $A$ is an algebra with zero multiplication.
\end{proof}

		\bigskip
	
	%%%%%%%%%%%%%%%%%%%%%%%%%%%%%%%%%%%%%%%%%%%%%%%%

	\begin{ack}
This work has been supported by the grant the Russian Science Foundation, RSF 24-21-00102, https://rscf.ru/project/24-21-00102/.
	\end{ack}
	
\end{document}